\documentclass{amsart}
\usepackage{euscript,amsmath, amssymb, amsfonts, mathtools}
\usepackage{mathrsfs}
\usepackage{mathabx}
\usepackage{stmaryrd}
\usepackage{crossreftools}

\makeatletter
\@namedef{subjclassname@1991}{Subject}
\makeatother

\usepackage[colorlinks=true, allcolors=blue]{hyperref}

\usepackage{tikz-cd}

\numberwithin{equation}{section}
\usepackage{enumitem}   

\newtheorem{theorem}{Theorem}[section]
\newtheorem{lemma}[theorem]{Lemma}
\newtheorem{proposition}[theorem]{Proposition}
\newtheorem{corollary}[theorem]{Corollary}

\theoremstyle{definition}
\newtheorem{definition}[theorem]{Definition}

\theoremstyle{remark}
\newtheorem*{remark}{Remark}

\usepackage{graphicx}  
\usepackage{icomma}
\usepackage{bbm}
\usepackage{tikz}
\usetikzlibrary{shapes}
\usepackage{color}
\usepackage{algorithmic}
\usepackage{verbatim}
\usepackage{xurl}
\usepackage{indentfirst}
\usepackage{lipsum}
\usepackage[backend = biber]{biblatex}
\usepackage{placeins}
\usepackage{bigints}
\usepackage{subcaption}
\usepackage{appendix}

\DeclareMathOperator{\Jac}{Jac}

\DeclareMathOperator{\CC}{\mathbb C}

\DeclareMathOperator{\ZZ}{\mathbb Z}

\DeclareMathOperator{\PP}{\mathbb P}

\DeclareMathOperator{\Per}{Per}

\usepackage{csquotes}

\begin{document}
	\title[Kleinian hyperelliptic funtions of weight 2]{Kleinian hyperelliptic funtions of weight 2 associated with curves of genus 2}
	
	\author{Matvey Smirnov}
	\address{119991 Russia, Moscow GSP-1, ul. Gubkina 8,
		Institute for Numerical Mathematics,
		Russian Academy of Sciences}
	\email{matsmir98@gmail.com}

	\begin{abstract}
		We introduce a new collection of special functions associated to a complex curve of genus 2 similar to Kleinian hyperelliptic $\sigma$-function. These functions are related to weight 2 $\theta$-functions in the same fashion as $\sigma$-function is related to the classical $\theta$-function. A key feature of the introduced functions is the fact that they are well-defined for genus 2 algebraic curves without any restrictions (in particular it is not needed to assume that the curve has a Weierstrass point at infinity).
		
	\end{abstract}
	\maketitle

    \section{Introduction}
	The study of higher genus Kleinian functions have recieved great attention in the span of last three decades. One of the main reasons for this attention is the convenience of the $\wp$-function approach to integrable systems (in particular, the KP equation). We refer to~\cite{BuchOldAndNew} and references therein for the comprehensive survey of recent developments in the theory of Kleinian functions. However, only quite a bit of research papers aim at numerical evaluation of Kleinian functions for genus $>1$. Noteworthy examples include~\cite{Bernatska1} and~\cite{Bernatska2}, where the approach is based on expressing $\wp$-functions through theta functions, and~\cite{Matsutani}, where $\wp$-functions are numerically evaluated via integration of the KdV equation. This contrasts drastically with the elliptic case, for which the effective numerical procedures are well known and understood for all types of special functions (see, e.g.~\cite{Smirnov},~\cite{Cremona},~\cite{labrandeg1},~\cite{Johansson},~and~\cite{luther}). 
	
	The most effective approach to calculations with elliptic functions is the Landen's (or, equivalently, AGM) method. Let us recall the main idea of this approach. The Landen's method involves construction of a chain of isogenies $C_0 \leftarrow C_1 \dots \leftarrow C_n$ of elliptic curves (obtained from the starting curve $C_0$ by ``doubling'' one of the periods iteratively). If the isogenous curves are chosen appropriately, then for sufficiently large $n$ the curve $C_n$ is ``close enough'' to degeneration and the elliptic functions, associated with $C_n$ can be evaluated by explicit formulas through elementary functions. Then the values of elliptic functions associated with $C_{k-1}$ are approximated using already computed values for $C_k$. This procedure for Weierstrass elliptic functions (i.e. the Kleinian functions of genus $1$) was described in~\cite{Smirnov}.
	
	We aim to give an algorithm to compute Kleinian functions of genus $2$ based on the ideas analogous to the Landen's method. Clearly, such an algorithm requires three main ingredients:
	\begin{enumerate}[label=(\alph*)]
		\item\label{ingra} A procedure that given a curve $C$ constructs another curve $\hat C$ with an isogeny $\Jac(C) \leftarrow \Jac(\hat C)$. It is required that iterations of this procedure yield a controllable degeneration of the curves.
		\item\label{ingrb} A formula that expresses Kleinian functions associated with $C$ through Kleinian functions associated with $\hat C$.
		\item\label{ingrc} An approximation for the Kleinian functions associated with curves ``near degeneration'' (i.e. the curves obtained by iterating the construction in~\ref{ingra}).
	\end{enumerate}
	Given these ingredients it is easy to formulate an algorithm that computes Kleinian functions of genus 2 similar to the described above Landen's method.
	
	For the part~\ref{ingra} there is a classical approach to constructing genus $2$ curves with isogenous Jacobian varieties due to Richelot (the Richelot isogeny in its modern presentation was proposed by Humbert~\cite{Humbert}). This construction was already used in calculation of periods of genus $2$ curves (see, e.g.~\cite{BostMestre} and~\cite{Chow}) and in isogeny-based cryptography (see, e.g.~\cite{CasselsFlynn},~\cite{flynnPaper}, ~\cite{takashima}, and references therein).   However, there is a problem with using Richelot isogenous curves in the ingredient~\ref{ingra}. That is, the definition of the Kleinian functions requires an additional restriction on the equation of the curve (namely, that there is only one point at infinity) that is not preserved by the construction of the isogenous curve. 
	
	In this paper we introduce a quadruple of, apparently, new special functions associated with a genus $2$ curve. The introduced functions are designed similarly to the $\sigma$-function, but squared. That is, they constitute a basis of the space of global sections of a linear bundle on Jacobian variety of a curve, which is the tensor square of the bundle of the $\sigma$-function. Similarly to the $\theta$-functional terminology adopted by Mumford in~\cite{mumfordI} we call these functions {\it{Kleinian functions of weight 2}}. We prove below that the classical genus 2 Kleinian functions (i.e. $\sigma, \zeta_j$, and $\wp_{jk}$) can be expressed through these functions and their derivatives. Thus, we are content with obtaining a computational procedure for Kleinian functions of weight $2$ instead of the classical Kleinian functions.
	
	An important feature of Kleinian functions of weight $2$ is that their definition does not require any conditions on the algebraic equation of the curve (in contrast to $\sigma$-function, which is usually defined only for hyperelliptic curves with Weierstrass point at infinity; there is an approach~\cite{ayanoNonWeier} to define $\sigma$-function for such curves, but it requires an additional choice of the branch point). Moreover, since the role of Kleinian functions of weight $2$ is similar to that of theta-functions of weight $2$, they define an embedding of the Kummer surface of a genus $2$ curve into $\CC\PP(3)$. This feature will allow us to find an equation that connects Kleinian functions of weight $2$ for Richelot isogenous curves, that is, to obtain the ingerdient~\ref{ingrb} of the algorithm. Finally, we note that for the introduced Kleinian functions of weight $2$ the problem~\ref{ingrc} also can be solved explicitly. These results (and the algorithm) will be the subject of forthcoming papers.
	
	The paper is organised as follows. In Section~\ref{secSpecialMerFunctions} we introduce some notation and derive several properties of $\wp$-functions. In Section~\ref{secKleinianWeight2} we introduce Kleinian functions of weight 2. We select four specific functions from the space of Kleinian functions of weight 2, and we derive their relations with classical Kleinian hyperelliptic functions. Finally, in Section~\ref{secTaylorExpansions} we find order 2 Taylor expansions of the introduced functions.
	\section{Meromorphic functions associated with curves of genus 2}\label{secSpecialMerFunctions}
	In this section we fix some notation and recall the definition and properties of $\wp$-functions associated with curves of genus $2$. Since in literature the curve is usually required to have only one point at infinity (and we are interested in the general case), we provide the proofs of all necessary results for completeness.
	\subsection{Notation} For brevity we shall denote the space of complex polynomials of degree $\le k$ by $\mathfrak P_k$. Given a polynomial $f$ we denote by $f_j$ the coefficient of $x^j$ in $f$. Throughout this section $f \in \mathfrak P_6$ denotes a complex polynomial of degree equal to $5$ or $6$ with simple roots. We shall call such polynomials {\it{admissible}}. We say that $f$ is in {\it{Weierstrass form}} if $f_6 = 0$ and $f_5 = 4$ (in literature often an additional condition is imposed, namely, $f_4 = 0$; we shall not require this condition). By $\mathcal X_f$ we denote the corresponding to the polynomial $f$ curve of genus $2$, i.e. the curve, whose affine part is given by the equation $y^2 = f(x)$ in the space of two variables $x,y$. Except the affine part this curve contains one point at infinity if $\deg f = 5$ and two points at infinity if $\deg f = 6$; these points are precisely the poles of the function $(x,y) \mapsto x$ (which we shall denote by simply $x$). Finally, let $\mathscr J_f:\mathcal X_f \to \mathcal X_f$ denote the {\it{hyperelliptic involution}} on $\mathcal X_f$, i.e. $\mathscr J_f(x,y) = (x,-y)$ on the affine part of $\mathcal X_f$. 
	
	By $\omega_1^f$ and $\omega_2^f$ we denote the standard holomorphic $1$-forms on $\mathcal X_f$ given by the formulas $\omega_1^f = dx/y$ and $\omega_2^f = xdx/y$. Let $\Per_f$ be the {\it{period lattice}} of the forms $\omega_1^f$, $\omega_2^f$ on $\mathcal X_f$, i.e. 
	\[\Per_f = \{w \in \mathbb C^2: \begin{pmatrix} w_1 \\ w_2 \end{pmatrix} = \int_\gamma \begin{pmatrix} \omega_1^f \\ \omega_2^f \end{pmatrix} \text{ for some } \gamma \in H_1(\mathcal X_f, \ZZ) \}. \]
	By $\Jac_f$ we denote the {\it{Jacobian variety}} of $\mathcal X_f$, i.e. $\Jac_f = \CC^2 / \Per_f$. Moreover, we let $\mathfrak L_f$ be the {\it{canonical hyperelliptic divisor class}} on $\mathcal X_f$ (i.e. the class of a divisor $P + \mathscr J_f(P)$ for any point $P \in \mathcal X_f$) and let $\mathscr A_f(D)$ denote the image of the divisor class $D - \mathfrak L_f$ in $\Jac_f$ under the Abel map for any effective divisor $D$ of degree $2$ on $\mathcal X_f$. We shall identify the effective degree $2$ divisors on $\mathcal X_f$ with the elements of the surface $\mathcal X_f^{(2)}$ (the symmetric Cartesian square of the curve $\mathcal X_f$). With this notation $\mathscr A_f$ is a bimeromorphic mapping of $\mathcal X_f^{(2)}$ onto $\Jac_f$. More precisely, $\mathscr A_f$ maps all divisors that belong to the class $\mathfrak L_f$ to zero, and bijectively maps $\mathcal X_f^{(2)} \setminus \mathfrak L_f$ to $\Jac_ f\setminus \{0\}$. Finally, we state an important property of the map $\mathscr A_f$, namely
	\begin{equation}\label{eqHyperellipticInvolutionAbel}
		\mathscr A_f[(\mathscr J_f(P)) + (\mathscr J_f(Q))] = - \mathscr A_f[(P) + (Q)]
	\end{equation}
	for all $(P) + (Q) \in \mathcal X_f^{(2)}$.
	
	\subsection{Meromorfic functions on \texorpdfstring{$\mathcal X_f^{(2)}$ and $\Jac_f$}.} Now we introduce the meromorphic functions $\xi^f_{11}$, $\xi_{12}^f$, and $\xi_{22}^f$ on the surface $\mathcal X_f^{(2)}$ via the formulas
	\begin{equation*}
		\xi^f_{22}(D) = x_1 + x_2,\;\;\xi^f_{12}(D) = -x_1x_2,\;\;\xi^f_{11}(D) = \frac{F_f(x_1,x_2) - 2y_1y_2}{4(x_1 - x_2)^2},
	\end{equation*}
	where $D = (x_1,y_1) + (x_2,y_2) \in \mathcal X_f^{(2)}$ and 
	\begin{equation*}
		F_f(a,b) = \\ 2f_0 + f_1(a + b) + 2f_2ab + f_3ab(a + b) + 2f_4a^2b^2 + f_5a^2b^2(a + b) + 2f_6a^3b^3.
	\end{equation*}
	Since $\mathscr A_f:\mathcal X_f^{(2)} \to \mathscr J_f$ is bimeromorphic there exist unique meromorphic functions $\wp^f_{11}$, $\wp^f_{12}$, and $\wp^f_{22}$ on $\Jac_f$ that satisfy the equations $\wp^f_{jk}(\mathscr A_f(D)) = \xi^f_{jk}(D)$ for a generic $D \in \mathcal X_f^{(2)}$. By an abuse of notation we shall use symbols $\wp^f_{jk}$ to denote both the functions on $\Jac_f$ defined above and the corresponding $\Per_f$-periodic functions on $\CC^2$. The introduced functions $\xi^f_{jk}$ and $\wp_{jk}^f$ are very well-known in the theory of algebraic curves of genus $2$ (see, e.g.~\cite{CasselsFlynn} and~\cite{baker}). 
	We outline the important properties of the functions $\xi^f_{jk}$ and $\wp^f_{jk}$ in the following proposition.
	\begin{proposition}\label{propPropertiesOfP}
		The following statements hold.
		\begin{enumerate}[label=(\roman*)]
			\item\label{ppi} The functions $\wp_{jk}^f$ are even. 
			\item\label{ppii} The meromorphic functions $1$, $\wp_{11}^f$, $\wp_{12}^f$, and $\wp_{22}^f$ are linearly independent.
			\item\label{ppiii} Let $(\infty_1) + (\infty_2)$ be the divisor of poles of the function $x$ on $\mathcal X_f$ (that is, $\infty_2 = \mathscr J_f(\infty_1)$ and, if $\deg f = 5$, then $\infty_1 = \infty_2$). Consider the irreducible subvarieties $A^f_j = \{D = (P) + (\infty)_j: P \in \mathcal X_f\}$ in $\mathcal X_f^{(2)}$ for $j = 1,2$. Then the divisors
			\[
			(\xi^f_{12}) + (A^f_1) + (A^f_2),\;\;(\xi_{22}^f) + (A^f_1) + (A^f_2)
			\]
			on $\mathcal X_f^{(2)}$ are effective.
			\item\label{ppiv} Consider the irreducible subvariety $L_f = \mathfrak L_f \cap \mathcal X_f^{(2)}$ in $\mathcal X_f^{(2)}$ that consists of special effective divisors of degree $2$ on $\mathcal X_f$. Then there exists $k \in \mathbb N$ such that the divisor
			\[(\xi_{11}^f) + (A^f_1) + (A^f_2) + k(L_f)\]
			on $\mathcal X_f^{(2)}$ is effective.
			\item\label{ppv} Let $B^f_j$ be the image of $A^f_j$ in $\Jac_f$ with respect to the map $\mathscr A_f$, $j = 1,2$. Then the divisors
			\[
			(\wp_{jk}^f) + (B^f_1) + (B^f_2)
			\]
			are effective for $(j,k) = (1,1),(1,2),(2,2)$.
		\end{enumerate}
	\end{proposition}
	\begin{proof}
		The statement~\ref{ppi} easily follows from~\eqref{eqHyperellipticInvolutionAbel} and the equalities
		\[
		\xi_{jk}^{(f)}\left[(P) + (Q)\right] = \xi_{jk}^{(f)}\left[(\mathscr J_f(P)) + (\mathscr J_f(Q))\right],
		\]
		which are evident from the definition of functions $\xi_{jk}^f$. To prove~\ref{ppii} note that $1, \xi^f_{22},$ and $\xi^f_{21}$ are linearly independent meromorphic functions on $\mathcal X_f^{(2)}$ (for instance, consider them on the diagonal). To prove that $\xi^f_{11}$ does not belong to the linear span of functions $1, \xi^f_{22}$, and $\xi_{12}^f$ consider arbitrary $x_0 \in \CC$ such that $f(x_0) \ne 0$. Let $u$ denote any branch of $\sqrt{f}$ defined in a neighbourhood of $x$. Finally, for small $t \ne 0$ consider divisor $D(t) = (x_0, -u(x_0)) + (x_0 + t, u(x_0 + t)) \in \mathcal X_f^{(2)}$, which obviously does not belong to the polar locus of functions $\xi_{jk}^f$ if $t \ne 0$. It is easy to see that $\xi_{22}^f(D(t))$ and $\xi_{12}^f(D(t))$ converge to $2x_0$ and $-x_0^2$ respectively when $t \to 0$. On the other hand $|\xi_{11}^f(t)| \to +\infty$ when $t \to 0$. Thus, $\xi_{11}^f$ cannot be a linear combination of $1, \xi_{22}^f$, and $\xi_{21}^f$. Thus, $1, \xi_{22}^f, \xi_{12}^f$, and $\xi_{11}^f$ are linearly independent and, since $\mathscr A_f$ is bimeromorphic, the statement~\ref{ppii} follows.
		
		The statements~\ref{ppiii} and~\ref{ppiv} are evident from the definition of functions $\xi_{jk}^f$, if we prove that $\xi_{11}^f$ does not have a pole along the analytic set $\mathrm{Diag} = \{(P) + (P): P \in \mathcal X_f\}$. To prove this observe that at a generic point $D \in \mathrm{Diag}$ the functions $\xi_{22}^f$ and $\xi_{12}^f$ constitute local coordinates on $\mathcal X_f^{(2)}$. Thus, $(x_1 - x_2)^2 = \left(\xi_{22}^f\right)^2 + 4\xi_{12}^f$ at a generic $D \in \mathrm{Diag}$ is the function defining the analytic set $\mathrm{Diag}$. Thus, to prove that $\xi_{11}^f$ does not have a pole along $\mathrm{Diag}$ it suffices to note that the numerator $F_f(x_1,x_2) - 2y_1y_2$ of the expression that defines $\xi_{11}^f$ has a zero along $\mathrm{Diag}$. The statement~\ref{ppv} now follows immediately from~\ref{ppiii} and~\ref{ppiv} since the analytic set $L$ is blown down by the map $\mathscr A_f$.
		
	\end{proof}
	
	We conclude this subsection with the quartic relation that is satisfied by $\wp_{jk}^f$.
	\begin{proposition}\label{propQuarticRelation}
		Let $f$ be an admissible polynomial. For convenience we denote by $\wp_{0}^f$ the function on $\Jac_f$ that is identically equal to $1$. Then there exist homogeneous polynomials of three variables $K_3$ and $K_4$ of degrees $3$ and $4$ respectively such that the relation
		\begin{equation}\label{eqQuartic}
			\left(\left(\wp_{22}^f\right)^2 + 4\wp_{0}^f\wp_{12}^f\right)\left(\wp_{11}^f\right)^2 + K_3\left(\wp_{0}^f, \wp_{12}^f, \wp_{22}^f\right)\wp_{11}^f + K_4\left(\wp_{0}^f, \wp_{12}^f, \wp_{22}^f\right) = 0
		\end{equation}
		holds.
	\end{proposition}
	The proof of Proposition~\ref{propQuarticRelation} and explicit formulas for $K_3$ and $K_4$ are given in~\cite[Section~3.1]{CasselsFlynn}. Note, however, that the relation is given there in terms of functions $\xi_{jk}^f$ (which does not affect the proof, since $\mathscr A_f$ is bimeromorphic) and our notation slightly differs from that in~\cite{CasselsFlynn}, so the corresponding adjustments are to be made to the expressions for $K_3$ and $K_4$. For curiosity we also give a determinantal form of the relation~\eqref{eqQuartic} (a special case for polynomials in Weierstrass form is given in~\cite{baker}). That is,~\eqref{eqQuartic} is equivalent to the equality
	\begin{equation}\label{eqQuarticDet}
		\det \begin{pmatrix}
			-f_0 & \dfrac{f_1}{2} & 2\wp_{11}^f & -2\wp_{12}^f\\
			\dfrac{f_1}{2} & -f_2 - 4\wp_{11}^f - f_6\left(\wp_{12}^f\right)^2 & \dfrac{f_3}{2} + \dfrac{f_5}{2}\wp_{12}^f + f_6 \wp_{12}^f\wp_{22}^f & 2\wp_{22}^f\\
			2\wp_{11}^f & \dfrac{f_3}{2} + \dfrac{f_5}{2}\wp_{12}^f + f_6 \wp_{12}^f\wp_{22}^f & -f_4 - f_5\wp_{22}^f - f_6 \left(\wp_{22}^f\right)^2 & 2\\
			-2\wp_{12}^f & 2\wp_{22}^f & 2 & 0
		\end{pmatrix} = 0.
	\end{equation}
	It can be verified that all terms of degree $\ge 5$ in $\wp_{jk}^f$ in the left-hand side of~\eqref{eqQuarticDet} cancel out, and the remaining terms add up to the polynomial given in~\cite[Section~3.1]{CasselsFlynn}.
	
	\section{Kleinian functions of weight 2}\label{secKleinianWeight2}
	
	In this section we introduce Kleinian functions of weight $2$. We derive all necessary properties of these functions and provide a way to express classical Kleinian functions $\sigma$, $\zeta_j$, $\wp_{jk}$ through them.
	
	\subsection{Periods of certain meromorphic forms of second kind.} Following Baker~\cite{baker} we introduce the meromorphic $1$-forms $r_1^f$ and $r_1^f$ of the second kind on $\mathcal X_f$ by the formulas
	\begin{equation}\label{eqRDef}
		r_1^f = \frac{f_3x + 2f_4x^2 + 3f_5x^3 + 4f_6x^4}{4y}dx,\;\;r_2^f = \frac{f_5x^2 + 2f_6x^3}{4y}dx.
	\end{equation}
	Let $\eta^f:\Per_f \to \mathbb C^2$ denote the additive group homomorphism defined by the condition
	\begin{equation}\label{eqEtaDef}
		\eta^f(w) = -\int_\gamma \begin{pmatrix} r_1^f \\ r_2^f\end{pmatrix},\text{ where } w = \int_\gamma \begin{pmatrix} \omega_1^f \\ \omega_2^f \end{pmatrix}.
	\end{equation}
	The homomorphism $\eta^f$ is well-defined, since $r_j^f$ is of the second kind (so the integral does not depend on the choice of the closed curve representing $\gamma$) and the mapping 
	\[
	H_1(\mathcal X_f, \ZZ) \to \Per_f,\;\;\gamma \mapsto  \int_\gamma \begin{pmatrix} \omega_1^f \\ \omega_2^f \end{pmatrix}
	\]
	is a group isomorphism. In order to state the main property of the homomorphism $\eta^f$ we need to introduce the intersection pairing $\langle \cdot,\cdot\rangle_{f}$ on $\Per_f$. It is defined as the bilinear form on $\Per_f$ canonically induced from the intersection pairing on $H_1(\mathcal X_f, \ZZ)$, i.e. $\langle w_1, w_2\rangle_f$ equals $\gamma_1 \circ \gamma_2$, where $\circ$ denotes the intersection pairing on $H_1(\mathcal X_f, \ZZ)$ and
	\[
	w_j = \int_{\gamma_j} \begin{pmatrix} \omega_1^f \\ \omega_2^f \end{pmatrix}, \;\;j = 1,2.
	\]
	The intersection pairing $\langle \cdot, \cdot \rangle_f$ is a non-degenerate antisymmetric $\ZZ$-bilinear form on $\Per_f$ and there exists a basis $a_1,a_2,b_1,b_2 \in \Per_f$ such that $\langle a_j, b_j\rangle_f = 1$, $j = 1,2$ and $\langle a_1, a_2 \rangle_f = \langle b_1, b_2 \rangle_f = \langle a_1, b_2 \rangle_f = \langle a_2 , b_1\rangle_f = 0$. Any quadruple of vectors that satisfies such properties is a basis, and we shall call such quadruple a {\it{symplectic basis}} in $\Per_f$. Now we are ready to state the main property of $\eta^f$, which can be viewed as a generalization of the classic Legendre identity~\cite[Theorem~4.2]{Chandra}. Consider arbitrary symplectic basis $a_1, a_2, b_1, b_2$ in $\Per_f$ and $2 \times 2$ matrices
	\begin{equation}\label{eqMatrices}
		A = \begin{pmatrix}
			a_1 & a_2
		\end{pmatrix},\;\;
		B = \begin{pmatrix}
			b_1 & b_2
		\end{pmatrix},\;\;
		\eta_A = \begin{pmatrix}
			\eta^f(a_1) & \eta^f(a_2)
		\end{pmatrix},\;\;
		\eta_B = \begin{pmatrix}
			\eta^f(b_1) & \eta^f(b_2)
		\end{pmatrix}.
	\end{equation}
	Then the identities
	\begin{equation}\label{eqLegendre}
		\begin{gathered}
			\eta_A^T B - A^T \eta_B = B \eta_A^T - A\eta_B^T= 2i\pi I, \\
			\left(\eta_A \eta_B^T\right)^T = \eta_A\eta_B^T,\;\; (\eta_A^TA)^T = \eta_A^TA,\;\;(\eta_B^TB)^T = \eta_B^TB
		\end{gathered}
	\end{equation}
	hold. The proof of these identities is the direct application of the Riemann's bilinear relations and is carried out in~\cite[\S~8]{baker} (note that our notation differs from that in~\cite{baker} by a factor of $2$). We also point out the important corollary of~\eqref{eqLegendre}, namely
	\begin{equation}\label{eqEtaIntValues}
		\eta^f(w)^Tv - \eta^f(v)^Tw \in 2\pi i \ZZ\;\;\forall v,w \in \Per_f.
	\end{equation}
	Clearly, it is sufficient to verify this fact for $v,w \in \{a_1,a_2,b_1,b_2\}$, and in this case it is evident from~\eqref{eqLegendre}.
	
	\subsection{Kleinian functions of weight 2.}
	Consider the linear space $\mathfrak S_f$ that consists of all holomorphic functions $\phi$ on $\mathbb C^2$ that satisfy the property
	\begin{equation}\label{eqSpaceDef}
		\phi(z + w) = \exp\left[2\eta(w)^T\left(z + \frac{w}{2}\right)\right]\phi(z)
	\end{equation}
	for all $z \in \CC^2$ and $w \in \Per_f$. We shall refer to elements of the space $\mathfrak S_f$ as {\it{Kleinian functions of weight}} $2$.
	
	We shall study the space $\mathfrak S_f$ using theta functions, so we fix some notation to the end of this subsection. At first choose arbitrary symplectic basis $a_1, a_2, b_1, b_2 \in \Per_f$ and consider corresponding elements $\alpha_1, \alpha_2, \beta_1, \beta_2$ of $H_1(\mathcal X_f, \ZZ)$. Consider the matrices $A$, $B$, $\eta_A$, $\eta_B$ defined as in~\eqref{eqMatrices}. Clearly, holomorphic $1$-forms $\phi_1, \phi_2$ on $\mathcal X_f$ defined by the equation
	\begin{equation}\label{eqNormalizedDifferentials}
		\begin{pmatrix}
			\phi_1 \\ \phi_2
		\end{pmatrix} = A^{-1}\begin{pmatrix}
			\omega^f_1 \\ \omega^f_2
		\end{pmatrix}
	\end{equation}
	are normalized with respect to cycles $\alpha_1$ and $\alpha_2$. Their period matrix $\Omega = A^{-1}B$ is a {\it{Riemann matrix}}, i.e. it is symmetric and its imaginary part is positive definite. Following Mumford~\cite[Definition~II.1.2]{mumfordI} we consider the space of {\it{weight 2}} $\theta$-{\it{functions}}, i.e. the space $R_2^\Omega$ that consists of all holomorphic functions $\phi$ on $\CC^2$ such that 
	\[
	\phi(z + n + \Omega m) = \exp\left[-2i\pi m^T \Omega m -4i\pi m^T z\right]\phi(z)
	\]
	for all $z \in \CC^2$ and all $m,n \in \ZZ^2$. 
	
	\begin{proposition}\label{propSpaceS}
		The space $\mathfrak S_f$ satisfies the following properties.
		\begin{enumerate}[label=(\roman*)]
			\item\label{pSi} Let $T(\phi)(z) = \exp\left[z^T (\eta_AA^{-1})z\right]\phi(A^{-1}z)$. Then $T$ is an isomorphism of the linear space $R_2^\Omega$ onto $\mathfrak S_f$.
			\item\label{pSii} All elements of $\mathfrak S_f$ are even functions and $\dim \mathfrak S_f = 4$.
			\item\label{pSiii} For all $z \in \CC^2$ there exists $\phi \in \mathfrak S_f$ such that $\phi(z) \ne 0$. Moreover, for $z,z' \in \CC^2$ the equation $\phi(z) = \phi(z')$ holds for all $\phi \in \mathfrak S_f$ if and only if at least one of the points $z + z'$, $z - z'$ belongs to $\Per_f$.
			\item\label{pSiv} For any even degree $2$ polynomial of two variables $p(z_1,z_2) = p_0 + p_{11}z_1^2 + p_{12}z_1z_2 + p_{22}z_2^2$ there is a unique function $\phi \in \mathfrak S_f$ such that $p$ coincides with order $2$ Taylor expansion of $\phi$ at zero, i.e. $\phi(z) = p(z) + \bar{o}(z^2)$.
		\end{enumerate}
	\end{proposition}
	\begin{proof}
		Consider any entire function $\phi$ of two variables. Note that if~\eqref{eqSpaceDef} holds for all $z$ for a given $w \in \Per_f$, then it also holds for $-w$ (and all $z$). Moreover, from~\eqref{eqEtaIntValues} it is easy to see that if~\eqref{eqSpaceDef} holds for $v,w \in \Per_f$, then it also holds for $v + w$. Thus, to verify that $\phi \in \mathfrak S_f$ it suffices to verify~\eqref{eqSpaceDef} for any basis in $\Per_f$. 
		
		Now we prove~\ref{pSi}. Assume that $\phi \in R^\Omega_2$ and consider $k \in \ZZ^2$. Then $w = Ak$ belongs to $\Per_f$ and using~\eqref{eqLegendre} we have
		\begin{multline*}
			T(\phi)(z + w) = \exp\left[(z + w)^T(\eta_A A^{-1})(z + w)\right]\phi(A^{-1}(z + w)) = \phi(A^{-1}z + k) \times \\ \exp\left[z^T(\eta_A A^{-1})z + z^T(\eta_A A^{-1})Ak + k^TA^T(\eta_A A^{-1})z + k^TA^T(\eta_A A^{-1})Ak\right] = \\
			\exp\left[z^T(\eta_A A^{-1})z + 2z^T\eta_Ak  + k^TA^T\eta_A k\right]\phi(A^{-1}z) = \\ \exp(2z^T\eta(w) + w^T\eta(w))T(\phi)(z).
		\end{multline*}
		Now let $w = Bk$ and similarly write
		\begin{multline*}
			T(\phi)(z + w)=\exp\left[(z + w)^T(\eta_A A^{-1})(z + w)\right]\phi(A^{-1}(z + w)) =\phi(A^{-1}z + \Omega k)\times \\ \exp\left[z^T(\eta_A A^{-1})z + z^T(\eta_A A^{-1})Bk + k^TB^T(\eta_A A^{-1})z + k^TB^T(\eta_A A^{-1})Bk\right] = \\
			\phi(A^{-1}z)\exp\left[-2\pi i k^T A^{-1}Bk - 4\pi i k^T A^{-1}z + z^T (\eta_A A^{-1})z\right] \times \\
			\exp\left[ z^T \left(A^T\right)^{-1}\eta_A^TBk + k^T(\eta_B^T A + 2\pi i I)A^{-1}z + k^T(\eta_B^T A + 2\pi i I)A^{-1}Bk \right] = \\
			\phi(A^{-1}z)\exp\left[-2\pi i k^T A^{-1}z + z^T (\eta_A A^{-1})z + z^T \left(A^T\right)^{-1}(A^T \eta_B + 2\pi i I)k\right] \times \\ \exp\left[\eta^f(w)^Tz + \eta^f(w)^Tw\right] = T(\phi)(z) \exp\left[2\eta^f(w)^Tz + \eta^f(w)^Tw\right].
		\end{multline*}
		Thus, if $\phi \in R_2^\Omega$, then $T(\phi) \in \mathfrak S_f$. We omit a similar calculation that shows the inverse implication, i.e. $T(\phi) \in \mathfrak S_f \Rightarrow \phi \in R_2^\Omega$. Thus, the statement~\ref{pSi} is proved.
		
		The statements~\ref{pSii} and~\ref{pSiii} follow from the corresponding properties of the space $R_2^\Omega$ (for~\ref{pSii} see~\cite[Proposition~II.1.3]{mumfordI} and for~\ref{pSiii} see~\cite[Theorem~II.1.3]{mumfordI}).
		
		Now we prove~\ref{pSiv}. Let $a(z) = a_1z_1 + a_2z_2$ be any non-zero linear form on $\CC^2$. At first we prove that there exists $\phi \in \mathfrak S_f$ such that $\phi(z) = a^2(z) + \bar{o}(z^2)$. Let $\tilde a(z) = a(Az)$. Since $\Omega$ is a period matrix of a curve of genus $2$ it is clear that theta divisor $\Theta_{\Omega}$ is a connected submanifold in $\CC^2$ (by~\cite[Theorem~I.3.1]{mumfordI} it coincides modulo translation with the image of the curve under the Abel map). Moreover, it is known that the Gauss map of the theta divisor is dominant (see, e.g.~\cite[Proposition~4.4.2]{LangeBirkenhake}), therefore there exists $q \in \Theta_\Omega$ such that $T_q\Theta_\Omega = \ker \tilde a$. That is, we have the expansion $\theta(z - q; \Omega) = c\tilde{a}(z) + \bar{o}(z)$ at zero with a constant $c \ne 0$. Since theta function is even, we conclude that $\phi(z) = \theta(z - q;\Omega)\theta(z + q, \Omega) = -c^2\tilde a(z)^2 + \bar{o}(z^2)$. Clearly, $\phi \in R_2^\Omega$ and $\tilde a(A^{-1}z)^2 = a(z)^2$ is the order $2$ Taylor expansion of $-c^{-2}T(\phi)$ at $0$. Since squares of linear forms generate the vector space of homogeneous polynomials of order $2$ it follows that for any $p(z) = p_{11}z_1^2 + p_{12}z_1z_2 + p_{22}z_2^2$ there exists $\phi \in \mathfrak S_f$ such that $\phi(z) = p(z) + \bar{o}(z^2)$ at zero. From~\ref{pSiii} it follows that $\mathfrak S_f$ contains functions that do not vanish at zero. Thus, for any even degree $2$ polynomial $p$ there exists $\phi \in \mathfrak S_f$ such that $p$ is the order $2$ Taylor expansion of $\phi$. Since the dimension of the space of degree $\le2$ even polynomials of two variables is $4$ (so it is equal to $\dim \mathfrak S_f$), the statement~\ref{pSiv} is proved.
	\end{proof}
	\begin{definition}
		For an admissible polynomial $f$ we call the fundamental Kleinian function of weight $2$ the unique element of $\mathfrak S_f$, whose Taylor expansion at zero equals $z_1^2 + \bar{o}(z^2)$. We shall denote it by $S^f$.
	\end{definition}
	
	
	All important properties of the function $S^f$ we obtain from an explicit expression of it through theta function.
	
	\begin{proposition}\label{propFormulaS}
		Let $\infty_1$ be one of the points on $\mathcal X_f$ introduced in Proposition~\ref{propPropertiesOfP}~\ref{ppiii}. Let $\Delta$ denote the Riemann's constant vector corresponding to the basis of cycles $\alpha_1, \alpha_2, \beta_1,\beta_2$ and $\infty_1$ as the selected point (see~\cite[Theorem~II.3.1]{mumfordI}).
		Then there exists a non-zero constant $c$ such that the equality
		\begin{equation}\label{eqSFormula}
			S^f(z) = c\exp\left[z^T (\eta_AA^{-1})z\right]\theta(A^{-1}z - \Delta; \Omega)\theta(A^{-1}z + \Delta; \Omega)
		\end{equation}
		holds for all $z \in \mathbb C^2$.
	\end{proposition}
	\begin{proof}
		Let $\phi(z) = \theta(z - \Delta;\Omega)\theta(z + \Delta;\Omega)$. Clearly, $\phi \in R_2^\Omega$, so $T(\phi) \in \mathfrak S_f$ by Proposition~\ref{propSpaceS}~\ref{pSi}. So we need to find the order $2$ Taylor expansion of $\phi$ at $0$. From the definition of $\Delta$ it is clear that $\Delta \in \Theta_\Omega$ and $T_\Delta\Theta_\Omega$ is spanned by the pushforward of a non-zero tangent vector $v \in T_{\infty_1}\mathcal X_f$ with respect to the Abel map (defined via normalized $1$-forms $\phi_1$ and $\phi_2$, see ~\eqref{eqNormalizedDifferentials}). Thus, $T_\Delta\Theta_\Omega$ is spanned by the second column of $A^{-1}$ (to see this use~\eqref{eqNormalizedDifferentials} and the fact that at infinity $\omega_1^f$ has a zero and $\omega_2^f$ does not). Let $a(z) = z_1$ and $\tilde a(z) = a(Az)$. It is clear from previous considerations that $T_\Delta\Theta_\Omega = \ker \tilde a$. Thus, there exists $c \ne 0$ such that $\tilde a(z)^2$ is the order $2$ Taylor expansion of $c\phi$. Then, $cT(\phi)$ has order $2$ Taylor expansion at zero equal to $a(z)^2 = z_1^2$. By uniqueness proved in Proposition~\ref{propSpaceS}~\ref{pSiv} the function $cT(\phi)$ coincides with $S^f$. On the other hand $cT(\phi)$ clearly coincides with the right-hand size of~\eqref{eqSFormula}.
	\end{proof}
	\begin{corollary}\label{corDivisorS}
		Let $B^f_1$ and $B^f_2$ denote the submanifolds in $\Jac_f$ defined in Proposition~\ref{propPropertiesOfP}~\ref{ppv}. Let $\tilde{B}^f_j$ be the preimage of $B^f_j$ with respect to the quotient mapping $\CC^2 \to \Jac_f$. Then $(S^f) = (\tilde{B}_1^f) + (\tilde{B}_2^f)$.
	\end{corollary}
	\begin{proof}
		This immediately follows from the definition of $\Delta$ (i.e.~Theorem~II.3.1 from~\cite{mumfordI}) and Proposition~\ref{propFormulaS}.
	\end{proof}
	\begin{corollary}\label{corBasisOfSpaceS}
		The functions $S^f, \wp_{11}^fS^f, \wp_{12}^fS^f, \wp_{22}^fS^f$ are entire holomorphic functions on $\CC^2$, and they constitute a basis in the space $\mathfrak S_f$. 
	\end{corollary}
	\begin{proof}
		The combination of Corollary~\ref{corDivisorS} and Proposition~\ref{propPropertiesOfP}~\ref{ppv} implies that function $\wp_{jk}^fS^f$ are entire, hence, they belong to $\mathfrak S_f$. Proposition~\ref{propPropertiesOfP}~\ref{ppii} implies that the functions $S^f, \wp_{11}^fS^f, \wp_{12}^fS^f, \wp_{22}^fS^f$ are linearly independent and by Proposition~\ref{propSpaceS}~\ref{pSii} they form a basis in $\mathfrak S_f$.
	\end{proof}
	
	\begin{definition}
		For an admissible polynomial $f$ we define $S^f_{jk} = \wp_{jk}^fS^f$, where $(j,k) = (1,1), (1,2), (2,2)$. We shall refer to these functions and the function $S^f$ as canonical Kleinian functions of weight 2.
	\end{definition}
	
	\subsection{Relation to the Kleinian sigma function of genus 2.}
	
	We conclude the section by a brief discussion of the connection between classical Kleinian functions of genus 2 and the functions defined above. Throughout this subsection we assume that $f$ is an admissible polynomial in Weierstrass form. Recall~\cite{Buch} that hyperelliptic $\sigma$-function for the curve $\mathcal X_f$ is defined by the formula
	\begin{equation}\label{eqSigmaDef}
		\sigma^f(z) = c\exp\left[\frac{1}{2}z^T(\eta_AA^{-1})z\right]\theta(A^{-1}z - \Delta; \Omega),
	\end{equation}
	where $A$ and $\eta_A$ are defined as in the previous subsection, and $\Delta$ is the same as in Proposition~\ref{propFormulaS}. The constant $c$ is chosen in order to satisfy the equality
	\begin{equation}\label{eqSigmaAtZero}
		\frac{\partial \sigma^f}{\partial z_1}(0,0) = 1.
	\end{equation}
	The function $\sigma^f$ does not depend on the choice of symplectic basis in $\Per_f$, as is proved in~\cite{Buch}. It is noteworthy, that there is an easy argument to prove this using the function $S^f$. At first note that $\left(\sigma^f\right)^2 = S^f$, as we show below in Proposition~\ref{propSigmaProperties}~\ref{pSigmai} (we don't need the uniqueness of $\sigma^f$ to prove this fact, only the formulas~\eqref{eqSigmaDef} and~\eqref{eqSigmaAtZero}). It follows that $\sigma^f$ is defined uniquely up to the sign change. Finally, the sign is determined by~\eqref{eqSigmaAtZero}. 
	
	The following proposition outlines fundamental properties of the function $\sigma^f$.
	
	\begin{proposition}\label{propSigmaProperties}
		Let $f$ be an admissible polynomial in Weierstrass form. Then the following statements hold.
		\begin{enumerate}[label=(\roman*)]
			\item\label{pSigmai} The function $\sigma^f$ is odd and $\left(\sigma^f\right)^2 = S^f$.
			\item\label{pSigmaii} For $(j,k) = (1,1), (1,2),(2,2)$ the equality
			\begin{equation}\label{eqLogDerP}
				\frac{\partial^2}{\partial z_j \partial z_k} \ln \sigma^f = -\wp_{jk}^f
			\end{equation}
			holds.
			\item\label{pSigmaiii} The function $\sigma^f$ satisfies the addition formula
			\[
			\frac{\sigma^f(u + v)\sigma^f(u - v)}{\sigma^f(u)^2\sigma^f(v)^2} = \wp_{22}^f(u)\wp_{12}^f(v) - \wp_{22}^f(v)\wp_{12}^f(u) + \wp_{11}^f(v) - \wp_{11}^f(u).
			\]
		\end{enumerate}
	\end{proposition}
	\begin{proof}
		To prove the statement~\ref{pSigmai} note that for polynomials $f$ such that $\deg f = 5$ the vector $\Delta$ from Proposition~\ref{propFormulaS} is an odd half-period, so $\theta(z - \Delta;\Omega) = -\theta(z + \Delta;\Omega)$. Thus, by~\eqref{eqSFormula} we get that 
		\[
		S^f = c\exp[z^T(\eta_A A^{-1})z]\theta(A^{-1}z - \Delta;\Omega)^2
		\]
		with a suitable constant $c$. From~\eqref{eqSigmaDef} and~\eqref{eqSigmaAtZero} now it is evident that $\left(\sigma^f\right)^2 = S^f$. The fact that $\sigma^f$ is odd is obvious either from the definition~\eqref{eqSigmaDef}, or by virtue of the fact that it is necessarily either odd, or even (since its square is even) and condition~\eqref{eqSigmaAtZero} is possible only if $\sigma^f$ is odd.
		
		The statements~\ref{pSigmaii} and~\ref{pSigmaiii} are proved in~\cite[\S~11]{baker} and~\cite[\S~30]{baker} respectively (see also~\cite{Buch} for generalizations for higher genera).
	\end{proof}
	The formula~\eqref{eqLogDerP} explains the choice of the subscripts in the notation $\wp_{jk}^f$. Following Baker~\cite{baker} for the polynomial $f$ in Weierstrass form we shall denote by $\wp_{jkl}^f$ the third logarithmic derivatives with respect to corresponding variables, i.e.
	\[
	\wp_{jkl}^f = -\frac{\partial^3}{\partial z_j\partial z_k \partial z_l} \ln \sigma^f.
	\]
	Similar notation is adopted also for derivatives of order $\ge 4$ but we shall not use them. For the first logarithmic derivatives of $\sigma^f$ the letter $\zeta$ is used, i.e. $\zeta_j^f = \partial \ln \sigma^f / \partial z_j$.
	
	\begin{remark}
		It should be noted that for polynomials $f$ not in Weierstrass form there is no function $\phi$ such that the equations~\eqref{eqLogDerP} are satisfied with $\phi$ instead of $\ln \sigma^f$. Indeed, by using the method of differentiation described below in Proposition~\ref{propDiffS} it can be checked that in general \[
		\frac{\partial \wp_{11}^f}{\partial z_2} \ne \frac{\partial \wp_{12}^f}{\partial z_1}.
		\]
	\end{remark}
	
	We conclude the subsection by showing that classical Kleinian functions can be expressed through Kleinian functions of weight $2$ and their first derivatives. For the functions $\wp_{jk}^f$ and $\zeta_j^f$ such an expression is obvious. In order to express the function $\sigma^f$ we use the {\it{duplication formula}} for $\sigma^f$.
	\begin{proposition}\label{propDuplication}
		Let $f$ be an admissible polynomial in Weierstrass form. Then the equality
		\begin{multline}\label{eqDuplication}
			\sigma^f(2z) = \sigma^f(z)^4(\wp_{122}^f(z)\wp_{12}^f(z) - \wp_{22}^f(z)\wp_{112}^f(z) - \wp_{111}^f(z)) = \\ S^f(z)^2(\wp_{122}^f(z)\wp_{12}^f(z) - \wp_{22}^f(z)\wp_{112}^f(z) - \wp_{111}^f(z))
		\end{multline}
		holds for a generic $z \in \CC^2$.
	\end{proposition}
	\begin{proof}
		Fix generic $u = \begin{pmatrix} u_1 & u_2 \end{pmatrix}^T$ and consider $v_t = \begin{pmatrix} u_1 + t & u_2\end{pmatrix}^T$. Application of the addition formula from Proposition~\ref{propSigmaProperties}~\ref{pSigmaiii} to $u$ and $v_t$ and passing to the limit $t \to 0$ yield~\eqref{eqDuplication} (note that $\sigma(u - v_t)/t \to -1$ when $t \to 0$).
	\end{proof}
	
	\begin{corollary}\label{corBetterDuplication}
		Let $f$ be an admissible polynomial in Weierstrass form. Then the equality
		\begin{equation}\label{eqDuplicationBetter}
			\sigma^f(2z) = S_{12}^f(z)\frac{\partial S_{22}^f}{\partial z_1}(z) - S_{22}^f(z)\frac{\partial S_{12}^f}{\partial z_1}(z) + S_{11}^f(z)\frac{\partial S^f}{\partial z_1}(z) - S^f(z)\frac{\partial S_{11}^f}{\partial z_1}(z)
		\end{equation}
		holds for all $z \in \CC^2$.
	\end{corollary}
	\begin{proof}
		The direct calculation shows that \[
		\begin{gathered}
			S_{12}^f(z)\frac{\partial S_{22}^f}{\partial z_1}(z) - S_{22}^f(z)\frac{\partial S_{12}^f}{\partial z_1}(z) = S^f(z)^2\left(\wp_{122}^f(z)\wp_{12}^f(z) - \wp_{22}^f(z)\wp_{112}^f(z)\right),\\
			S^f(z)\frac{\partial S_{11}^f}{\partial z_1}(z) - S_{11}^f(z)\frac{\partial S^f}{\partial z_1}(z) = S_f(z)^2 \wp_{111}^f(z).
		\end{gathered}
		\]
		The substitution of these identities into~\eqref{eqDuplication} yields~\eqref{eqDuplicationBetter}.
	\end{proof}
	
	\section{Taylor expansions of Kleinian functions of weight 2}\label{secTaylorExpansions}
	
	In this section we calculate order $2$ Taylor expansions of functions $S_{jk}^f$. These expansions will be crucial for our further developments. At first let us formulate the final result.
	
	\begin{theorem}\label{thTaylorExpansions}
		Let $f$ be an admissible polynomial. Then the expansions
		\begin{equation}\label{eqTaylorExpansions}
			S_{22}^f(z) = 2z_1z_2 + \bar{o}(z^2),\;\;
			S_{12}^f(z) = -z_2^2 + \bar{o}(z^2),\;\;
			S_{11}^f(z) = 1 + \bar{o}(z^2)
		\end{equation}
		hold.
	\end{theorem}
	
	We prove Theorem~\ref{thTaylorExpansions} in several steps.
	
	\begin{lemma}\label{lemValuesAtZero}
		Let $f$ be an admissible polynomial. Then $S_{12}^f(0) = S_{22}^f(0) = 0$ and $S_{11}^f(0) \ne 0$.
	\end{lemma}
	\begin{proof}
		The scheme of the proof is quite similar to that of Proposition~\ref{propPropertiesOfP}~\ref{ppii}. Consider arbitrary $x_0 \in \CC$ such that $f(x_0) \ne 0$ and consider any single-valued branch $u$ of $\sqrt{f}$ defined in a neighbourhood of $x_0$. With this choice of the square root branch we define divisor $D(t) = \left(x_0, -u(x_0)\right) + \left(x_0 + t, u(x_0)+t\right) \in \mathcal X_f^{(2)}$ for small $t \ne 0$. By definition $\mathscr A_f(D(t))$ can be represented by the vector
		\[
		A(t) = \begin{pmatrix}
			\displaystyle \int_{x_0}^{x_0 + t}\frac{dx}{u(x)}\\ \displaystyle\int_{x_0}^{x_0 + t}\frac{xdx}{u(x)}
		\end{pmatrix},
		\]
		which converges to $0$ when $t \to 0$. Clearly, $S(A(t)) \to 0$, as $t \to 0$. It remains to note that $\wp_{21}^f(\mathscr A_f(D(t)))$ and $\wp_{22}^f(\mathscr A_f(D(t)))$ converge to $-x_0^2$ and $2x_0$ respectively as $t \to 0$. By taking product we obtain that $S_{12}^f(A(t)), S_{22}^f(A(t))  \to 0$ as $t \to 0$. Thus, $S_{12}^f(0) = S_{22}^f(0) = 0$. Clearly, $S_{11}^f(0)$ cannot be equal to zero, since otherwise there is no function $\phi \in \mathfrak S_f$ such that $\phi(0) \ne 0$ by Corollary~\ref{corBasisOfSpaceS} contradicting Proposition~\ref{propSpaceS}~\ref{pSiii}.
	\end{proof}
	
	\begin{lemma}\label{lemSP12expansion}
		Let $f$ be an admissible polynomial. Then there exists $c \in \CC$ such that the expansion
		\begin{equation}\label{eqSP12expansion}
			S_{12}^f(z) = az_2^2 + \bar{o}(z^2)
		\end{equation}
		holds.
	\end{lemma}
	\begin{proof}
		Let $(P_1) + (P_2) \in \mathcal X_f^{(2)}$ be the divisor of zeros of the function $x$ (that is, $P_1$ and $P_2$ are the points of $\mathcal X_f$ with zero $x$-coordinate; if $f(0) = 0$, then $P_1 = P_2$ is a Weierstrass point). Let $C_j = \{\mathscr A_f[(P_j) + (P)]: P \in \mathcal X_f\} \subset \Jac_f$. Finally, let $\tilde C_j$ be the preimage of $C_j$ with respect to the quotient mapping $\CC^2 \to \Jac_f$. The definition of $\wp_{12}^f$ implies that $\left(S_{12}^f\right) = (C_1) + (C_2)$. Now we note that $\tilde C_j$ is a submanifold in $\CC^2$, $0 \in \tilde C_j$, and $T_0\tilde C_j$ is spanned by the pushforward of a tangent vector to $\mathcal X_f$ at the point $P_j$ with respect to the Abel map. Therefore, $T_0 \tilde C_j$ is spanned by $\begin{pmatrix} 1 & 0 \end{pmatrix}^T$ for $j = 1,2$ (this follows from the fact that $\omega_2^f$ has a zero at $P_j$, while $\omega_1^f$ does not). Consider defining functions $v_1, v_2$ at zero of submanifolds $C_1,C_2$ respectively, i.e. they are defined in a neighbourhood of zero $U \subset \CC^2$ such that $z \in U \cap C_j \Leftrightarrow v_j(z) = 0$ and $dv_j(z) \ne 0$ for all $z \in U$ such that $v_j(z) = 0$. From the description of $T_0\tilde C_j$ follows that $v_j(z) = \alpha_j z_2 + \bar{o}(z)$, where $\alpha_j \ne 0$. Finally, the germ of $v_1v_2$ at zero divides the germ of $S_{12}^f$, for $(S_{12}^f) = (C_1) + (C_2)$. The expansion~\eqref{eqSP12expansion} follows.
	\end{proof}
	\begin{remark}
		Along the lines of the proof of Lemma~\ref{lemSP12expansion} it is easy to derive a formula for $S_{12}^f(z)$ in terms of theta functions. The answer is the same as in the formula~\eqref{eqSFormula}, but $\Delta$ should be replaced by the Riemann constant vector corresponding to one of the points $P_j$ instead of points at infinity.
	\end{remark}
	\begin{corollary}\label{corPartialExpansions}
		Let $f$ be an admissible polynomial. Then there exists $c \in \CC$ such that the expansions
		\begin{equation}\label{eqSP12and22expansion}
			\begin{split}
				S_{22}^f(z) & =  cz_1z_2 + \bar{o}(z^2) \\
				S_{12}^f(z) & =  -\frac{c^2}{4}z_2^2 + \bar{o}(z^2)
			\end{split}
		\end{equation}
		hold.
	\end{corollary}
	\begin{proof}
		Consider polynomials $K_3$ and $K_4$ from Proposition~\ref{propQuarticRelation}. Since the relation~\eqref{eqQuartic} is homogeneous of order $4$ we can multiply it by $\left(S^f\right)^4$ and obtain the relation
		\begin{equation}\label{eqRelationForS}
			\left(\left(S_{22}^f\right)^2 + 4S^fS_{12}^f\right)\left(S_{11}^f\right)^2 = - K_3\left(S^f, S_{12}^f, S_{22}^f\right)S_{11}^f - K_4\left(S^f, S_{12}^f, S_{22}^f\right).
		\end{equation}
		Lemma~\ref{lemValuesAtZero} implies that the right-hand side in~\eqref{eqRelationForS} has zero of order at least $6$ at $z = 0$ (note that all functions are even). Therefore, the function \[
		\left(S_{22}^f\right)^2 + 4S^fS_{12}^f
		\] 
		also has zero of order at least $6$ at $z = 0$.  From the expansions
		\[
		\begin{gathered}
			S^f(z) = z_1^2 + \bar{o}(z^2),\;\;S_{12}^f(z) = az_2^2 + \bar{o}(z^2),\\S_{22}^f(z) = p_{11}z_1^2 + p_{12}z_1z_2 + p_{22}z_2^2 + \bar{o}(z^2),
		\end{gathered}
		\]
		it follows that 
		\[
		\left(S_{22}^f\right)^2 + 4S^fS_{12}^f = (p_{11}z_1^2 + p_{12}z_1z_2 + p_{22}z_2^2)^2 + 4az_1^2z_2^2 + \bar{o}(z^4).
		\]
		Since the left-hand side vanishes at zero with order $6$ we get that
		\[
		(p_{11}z_1^2 + p_{12}z_1z_2 + p_{22}z_2^2)^2 + 4az_1^2z_2^2 = 0.
		\]
		Therefore, $p_{11} = p_{22} = 0$ and $p_{12}^2 + 4a = 0$. Thus,~\eqref{eqSP12and22expansion} holds with $c = p_{12}$.
	\end{proof}
	
	To proceed further we need to differentiate Kleinian functions of weight 2. The method we present below is derived from~\cite{baker}, where the calculations are carried out for polynomials in Weierstrass form. At first we introduce the meromorphic functions $\rho^f_1$ and $\rho^f_2$ on $\CC^2$ defined as the integrals of the $1$-forms $r_1^f$ and $r_2^f$ respectively. That is, if $z \in \CC^2$ coincides with $\mathscr A_f(D)$ modulo $\Per_f$, where $D = (x_1,y_1) + (x_2,y_2)$ is non-special divisor, then there exists a curve $\gamma$ on $\mathcal X_f$ that connects $(x_1, -y_1)$ to $(x_2, y_2)$ and \[
	z = \int_\gamma \begin{pmatrix}
		\omega_1^f \\ \omega_2^f
	\end{pmatrix}.
	\]
	Clearly, $\gamma$ is defined up to adding cycles that are homologous to zero, thus, if both points $(x_1,y_1)$ and $(x_2,y_2)$ are not at infinity, then the integrals $\rho^f_j(z) = \int_\gamma r_1^f$, $j = 1,2$ are well-defined and do not depend on the choice of $\gamma$. Clearly, the foregoing construction defines $\rho_j^f$ on the set $\CC^2 \setminus (\tilde B^f_1 \cup \tilde B^f_2)$, where submanifolds $\tilde B^f_j \subset \CC^2$ were defined in Corollary~\ref{corDivisorS}.
	
	\begin{lemma}\label{lemRhoProperties}
		Let $f$ be an admissible polynomial. Then the following statements hold.
		\begin{enumerate}[label=(\roman*)]
			\item\label{Rhoi} Functions $\rho_1^f$ and $\rho_2^f$ are odd meromorphic functions on $\CC^2$.
			\item\label{Rhoii} Let $\lambda^f$ denote the $\Per_f$-periodic meromorphic function on $\CC^2$ defined by the equation 
			\[
			\lambda^f(D) = \frac{y_1 - y_2}{x_1 - x_2}
			\]
			for $D = (x_1,y_1) + (x_2,y_2)$. Then $\lambda^f$ is odd and the functions
			\[
			\left(\rho_1^f - \frac{\lambda^f}{2}\right) S^f,\;\;\rho_2^fS^f
			\]
			are entire analytic functions on $\CC^2$.
			\item\label{Rhoiii} For all $w \in \Per_f$ the relations
			\[
			\rho_1^f(z + w) = -\eta^f_1(w) + \rho_1(z),\;\;\rho_2^f(z + w) = -\eta^f_2(w) + \rho_2^f(z)
			\]
			hold, where $\eta^f(w) = \begin{pmatrix}
				\eta^f_1(w) & \eta^f_2 (w)
			\end{pmatrix}^T$.
		\end{enumerate}
	\end{lemma}
	\begin{proof}
		The statement~\ref{Rhoiii} is obvious. Indeed, assume that \[
		z = \int_\gamma \begin{pmatrix}
			\omega_1^f \\ \omega_2^f
		\end{pmatrix},
		\]
		where $\gamma$ is a curve on $\mathcal X_f$ that connects $(x_1, -y_1)$ to $(x_2, y_2)$. Then $z + w$ is represented in the same way by the curve $\tilde \gamma$, which is obtained from $\gamma$ by adding a cycle that represents the period $w$. The corresponding integral of $r_k^f$ is by definition (i.e. by formula~\eqref{eqEtaDef}) equal to $-\eta_k^f(w)$, hence~\ref{Rhoiii} holds. To prove that $\rho_1^f$, $\rho_2^f$, and $\lambda^f$ are odd in view of~\eqref{eqHyperellipticInvolutionAbel} it suffices to apply the hyperelliptic involution and verify the sign change. Finally, we note that the statement~\ref{Rhoii} implies that $\rho_1^f$ and $\rho_2^f$ are meromorphic. Thus, it remains to prove~\ref{Rhoii}. For that we need some preparation.
		
		Let $\infty_1,\infty_2$ denote the points at infinity of the curve $\mathcal X_f$. As in Corollary~\ref{corDivisorS} let $B^f_j \subset \Jac_f$ be the set $B^f_j = \{\mathscr A_f[(P) + (\infty_j)]: P \in \mathcal X_f\}$ and let $\tilde B^f_j \subset \CC^2$ be the preimage of $B^f_j$ with respect to the quotient mapping $\CC^2 \to \Jac_f$. Now fix $j \in \{1,2\}$ and consider $z \in \tilde B_j^f$ such that $z = \mathscr A_f(D)\;\mathrm{mod}\;\Per_f$, where $D = (x,y) + (\infty_j)$ and $(x,y)$ is not at infinity. Clearly, a generic $z \in \tilde B_j^f$ can be represented in such way. Consider neighbourhoods $U, V \subset \mathcal X_f$ of points $(x,y)$ and $\infty_j$ respectively, and meromorphic functions $u_1,u_2$ on $U$ and $v_1,v_2$ on $V$ respectively such that $du_k = r_k^f|_U$ and $dv_k = r_k^f|_V$ for $k = 1,2$ (such neighbourhoods and functions exist because $r_k^f$ is of the second kind). Moreover, we can shrink $U$ and $V$ if needed to achieve that $U \cap V = \emptyset$, $U$ and $V$ are simply connected, and that for all $P \in U$ and $Q \in V$ the divisor $(P) + (Q)$ is non-special. In this case we have a biholomorphic mapping $A:U \times V \to W$, where $W \subset \CC^2$ is a neighbourhood of $z$ such that $A((x,y), \infty_j) = z$ and $A(P,Q)$ is a representative of $\mathscr A_f[(P) + (Q)]$ for all $P \in U$, $Q \in V$. Moreover, there exists a constant $c \in \CC$ such that for all $w \in W \setminus \tilde B_1^f$ we have $\rho_k^f(w) = u_k(P) + v_k(Q) + c$, where $A(P,Q) = w$. 
		
		With this preparation we now separately consider cases $k = 1$ and $k = 2$. At first note that $r_2^f$ has a pole of order $2$ at points $\infty_j$, $j = 1,2$ (if $\deg f = 5$, then $r_2^f$ has a single double pole at $\infty_1 = \infty_2$). This fact is evident from the definition~\eqref{eqRDef} of $r_2^f$. Thus, $v_2$ has a pole of order $1$ at $\infty_1$ and therefore, the function $\rho_2^f$ in $W$ has a pole of order $1$ along the submanifold $W \cap \tilde B_j^f$. Since there are no other poles of $v_2$ and $u_2$, we conclude that $\rho_2^fS^f$ is holomorphic in $W$. Thus, we have proved that $\rho_2^fS^f$ is holomorphic in a neighborhood of a generic $z \in \tilde{B}_j^f$, $j = 1,2$. Thus, $\rho_2^fS^f$ is holomorphic outside a discrete set, so by the Riemann removable singularity theorem (see, e.g.~\cite[Theorem~7.1.2]{Grauert}) we conclude that $\rho_2^f S^f$ is holomorphic everywhere. Thus, we have proved~\ref{Rhoii} for the function $\rho_2^f$.
		
		Now we consider the function $\rho_1^f$. Since the order of the pole of $r_1^f$ at infinity is $3$ we need to modify the previous argument. Consider a meromorphic function $g$ on $\mathcal X_f$ defined as $(x,y) \mapsto y/x$. By a straightforward calculation we get
		\[
		2r_1^f - dg = \frac{2f_0 + f_1x}{2x^2y}dx,            
		\]
		so $2r^f_1 - dg$ is holomorphic at infinity. Thus, the function $2v_1 - g$ is also holomorphic at infinity. Now let $l(P,Q) = (y_1 - y_2)/(x_1 - x_2)$, where $P = (x_1,y_1)$ and $Q = (x_2,y_2)$. A simple calculation shows that for any $P$ that is not at infinity the function $Q \mapsto l(P,Q) - g(Q)$ has a pole of order at most $1$ at infinity. Therefore, the function $(P,Q) \mapsto u_1(P) + v_2(Q) - l(P,Q)/2$ has a pole of order $1$ along the submanifold $\{Q = \infty_j\}$. Thus, $(\rho_1^f - \lambda^f/2)S^f$ is holomorphic in a neighbourhood of $z$. The rest of the proof is the same as for the function $\rho_2^f$.
	\end{proof}
	\begin{proposition}\label{propDiffS}
		Let $f$ be an admissible polynomial. Then the equalities
		\begin{equation}\label{eq1DiffS}
			\frac{\partial \ln S^f}{\partial z_1} = -2\rho_1^f + \lambda^f,\;\;\frac{\partial \ln S^f}{\partial z_2} = -2\rho_2^f
		\end{equation}
		\begin{equation}\label{eq2DiffS}
			\begin{gathered}
				\frac{\partial^2 \ln S^f}{\partial z_1^2} = -2\wp^f_{11} - f_6 \left(\wp_{12}^f\right)^2,\;\;\frac{\partial^2 \ln S^f}{\partial z_1 \partial z_2} = - \frac{f_5}{2}\wp_{12}^f - f_6\wp_{12}^f\wp_{22}^f,\\ \frac{\partial^2 \ln S^f}{\partial z_2^2} = - \frac{f_5}{2}\wp_{22}^f - f_6\left(\left(\wp_{22}^f\right)^2 + \wp_{12}^f\right)
			\end{gathered}
		\end{equation}
		hold.
	\end{proposition}
	Note that equalities~\eqref{eqLogDerP} can be obtained by simplifying~\eqref{eq2DiffS} and using Proposition~\ref{propSigmaProperties}~\ref{pSigmai} when $f$ is in Weierstrass form.
	\begin{proof}
		Consider the functions 
		\[
		a(z) = \frac{\partial \ln S^f}{\partial z_1} + 2\rho_1^f - \lambda^f,\;\;b(z) = \frac{\partial \ln S^f}{\partial z_2} + 2\rho_2^f.
		\]
		From Lemma~\ref{lemRhoProperties}~\ref{Rhoiii} and the definition~\eqref{eqSpaceDef} of the space $\mathfrak S_f$ it is easy to see that $a$ and $b$ are $\Per_f$-periodic. Moreover, $a$ and $b$ are odd, in view of Lemma~\ref{lemRhoProperties}~\ref{Rhoi} and the fact that $S^f$ is even. Moreover, $aS^f$ and $bS^f$ are entire holomorphic functions due to Lemma~\ref{lemRhoProperties}~\ref{Rhoii}. Thus, $aS^f$ and $bS^f$ belong to $\mathfrak S_f$. From Proposition~\ref{propSpaceS}~\ref{pSii} it now follows that $aS^f = bS^f = 0$, as there are no non-trivial odd elements in $\mathfrak S_f$. The equations~\eqref{eq1DiffS} follow, since they can be written as equalities $a = b = 0$.
		
		The equations~\eqref{eq2DiffS} follow from~\eqref{eq1DiffS} by a straightforward calculation proposed in~\cite{baker}); we only sketch the proof. For a non-special divisor $D = (x_1,y_1) + (x_2,y_2)$ the map $\mathscr A_f$ is biholomorphic in a suitable neighbourhood of $D$. Moreover, if $(x_1,y_1) \ne (x_2,y_2)$ and neither of the points in $D$ is a Weierstrass point or a point at infinity, then $(x_1,x_2)$ constitute local coordinates on $\mathcal X_f^{(2)}$ near $D$ and it is easy to calculate the Jacobi matrix of $\mathscr A_f$. Namely,
		\[
		\frac{\partial \mathscr A_f}{\partial x_1} = \begin{pmatrix}
			1/y_1 \\ x_1/y_1
		\end{pmatrix},\text{ and }\frac{\partial \mathscr A_f}{\partial x_1} = \begin{pmatrix}
			1/y_2 \\ x_2/y_2
		\end{pmatrix}.
		\]
		If we let $z = \begin{pmatrix} z_1 & z_2 \end{pmatrix}^T = \mathscr A_f(D)$, then we can calculate the Jacobi matrix of the inverse map $z \mapsto (x_1(z),x_2(z))$ by taking the inverse matrix. A simple calculation yields
		\[
		\frac{\partial x_1}{\partial z_1} = \frac{y_1x_2}{x_2 - x_1},\;\;\frac{\partial x_2}{\partial z_1} = \frac{y_2x_1}{x_1 - x_2},\;\; \frac{\partial x_1}{\partial z_2} = \frac{y_1}{x_1 - x_2},\;\;\frac{\partial x_2}{\partial z_1} = \frac{y_2}{x_2 - x_1}.
		\]
		Since this calculation is valid at a generic $z \in \CC^2$ we can calculate the derivatives of a function $\phi$ by the formulas
		\[
		\begin{gathered}
			\frac{\partial \phi}{\partial z_1} = \frac{\partial \phi}{\partial x_1}\frac{\partial x_1}{\partial z_1} + \frac{\partial \phi}{\partial x_2}\frac{\partial x_2}{\partial z_1} = \frac{\partial \phi}{\partial x_1}\frac{y_1x_2}{x_2 - x_1} + \frac{\partial \phi}{\partial x_2}\frac{y_2x_1}{x_1 - x_2}, \\ \frac{\partial \phi}{\partial z_1} = \frac{\partial \phi}{\partial x_1}\frac{\partial x_1}{\partial z_1} + \frac{\partial \phi}{\partial x_2}\frac{\partial x_2}{\partial z_1} = \frac{\partial \phi}{\partial x_1}\frac{y_1}{x_1 - x_2} + \frac{\partial \phi}{\partial x_2}\frac{y_2}{x_2 - x_1}
		\end{gathered}
		\]
		Applying this to~\eqref{eq1DiffS} yields~\eqref{eq2DiffS}.
	\end{proof}
	\begin{proof}[Proof of Theorem~\ref{thTaylorExpansions}]
		At first we calculate the value 
		\[
		\frac{\partial^2 S_{22}^f}{\partial z_1 \partial z_2}(0).
		\]
		At first write
		\begin{equation}\label{eqSecDer}
			\frac{\partial^2 S_{22}^f}{\partial z_1 \partial z_2} = \frac{\partial^2 S^f}{\partial z_1 \partial z_2}\wp_{22}^f + \frac{\partial S^f}{\partial z_1}\frac{\partial \wp_{22}^f}{\partial z_2} + \frac{\partial S^f}{\partial z_2}\frac{\partial \wp_{22}^f}{\partial z_1} + S^f\frac{\partial^2 \wp_{22}^f}{\partial z_1 \partial z_2}.
		\end{equation}
		We calculate the limit of each term when $z \to 0$. To do so
		consider arbitrary $x_0 \in \CC$ that is not a root of $f$ and let $u$ be any branch of $\sqrt{f}$ defined in a neighbourhood of $x$. Finally, for small $t \ne 0$ consider divisor $D(t) = (x_0, -u(x_0)) + (x_0 + t, u(x_0 + t)) \in \mathcal X_f^{(2)}$. By definition $\mathscr A_f(D(t))$ can be represented by a vector
		\[
		A(t) = \begin{pmatrix}
			\displaystyle\int_{x_0}^{x_0+t} \frac{dx}{u(x)} \\
			\displaystyle\int_{x_0}^{x_0+t} \frac{xdx}{u(x)}
		\end{pmatrix}.
		\]
		Therefore, as $t \to 0$ we have the expansions
		\[
		\begin{gathered}
			A_1(t) = \frac{t}{u(x_0)} + \bar{o}(t),\;\; A_2(t) = \frac{x_0 t}{u(x_0)} + \bar{o}(t),\\S^f(A(t)) = \frac{t^2}{f(x_0)} + \bar{o}(t^2),\;\;\wp_{22}^f(A(t)) = 2x_0 + \bar{o}(1)\\\frac{\partial S^f}{\partial z_1}(A(t)) = \frac{2t}{u(x_0)} + \bar{o}(t^2),\;\;\frac{\partial S^f}{\partial z_2}(A(t)) = \bar{o}(t^2),\;\;\frac{\partial^2 S^f}{\partial z_1\partial z_2}(A(t)) = \bar{o}(t).
		\end{gathered}
		\]
		To calculate the derivatives of $\wp_{22}^f$ it is easy to apply the method described in the proof of Proposition~\ref{propDiffS}. It provides the formulas
		\[
		\begin{gathered}
			\frac{\partial \wp_{22}^f}{\partial z_1}(z) = \frac{y_1x_2 - y_2x_1}{x_2 - x_1},\;\;\frac{\partial \wp_{22}^f}{\partial z_2}(z) = \frac{y_2 - y_1}{x_2 - x_1},\\
			\frac{\partial^2 \wp_{22}^f}{\partial z_1 \partial z_2}(z) = -\frac{f'(x_1)x_2 + f'(x_2)x_1}{2(x_2 - x_1)^2} + \frac{y_1y_2}{(x_2 - x_1)^2} + \frac{f(x_2)x_1 - f(x_1)x_2}{(x_2 - x_1)^3},
		\end{gathered}
		\]
		where $z = \mathscr A_f(D)$ and $D = (x_1,y_1) + (x_2,y_2)$. By substituting $D(t)$ into these formulas we obtain the expansions
		\[
		\begin{gathered}
			\frac{\partial \wp_{22}^f}{\partial z_1}(A(t)) = \frac{1}{t}(-2u(x_0)x_0 + \bar{o}(1)),\;\;\frac{\partial \wp_{22}^f}{\partial z_2}(A(t)) = \frac{1}{t}(2u(x_0) + \bar{o}(1)),\\
			\frac{\partial^2 \wp_{22}^f}{\partial z_1 \partial z_2}(A(t)) = \frac{1}{t^2}(-2f(x_0) + \bar{o}(1)).
		\end{gathered}
		\]
		Combining all the expansions for $S^f$, $\wp_{22}^f$ and their derivatives and substituting into~\eqref{eqSecDer} we get that
		\[
		\frac{\partial^2 S_{22}^f}{\partial z_1 \partial z_2}(A(t)) \to 2,\text{ when }t \to 0.
		\]
		Thus, from Corollary~\ref{corPartialExpansions} it follows that
		\[
		S_{22}^f(z) = 2z_1z_2 + \bar{o}(z^2),\;\;S_{12}^f(z) = -z_2^2 + \bar{o}(z^2).
		\]
		It remains to find the order $2$ Taylor expansion of $S_{11}^f$. The idea is to consider the expansion
		\[
		S^f(z) = z_1^2 + az_1^4 + bz_1^3z_2 + cz_1^2z_2^2 + dz_1z_2^3 + ez_2^4 + \bar{o}(z^4).
		\]
		Substituting it into the equations
		\[
		\begin{gathered}
			S^f\frac{\partial^2 S^f}{\partial z_2^2} - \left(\frac{\partial S^f}{\partial z_2}\right)^2 = -\frac{f_5}{2}S^fS^f_{22}  - f_6\left(\left(S_{22}^f\right)^2 + S^fS_{12}^f\right), \\
			S^f\frac{\partial^2 S^f}{\partial z_1 \partial z_2} - \frac{\partial S^f}{\partial z_1}\frac{\partial S^f}{\partial z_2} = -\frac{f_5}{2}S^fS^f_{12}  - f_6S_{12}^fS_{22}^f,
		\end{gathered}
		\]
		which follow from~\eqref{eq2DiffS}, yields $b = c = 0$, $d = -f_5/6$ and $e = -f_6/4$. Thus,
		\[
		S^f(z) = z_1^2 + az_1^4 - \frac{f_5}{6}z_1z_2^3 - \frac{f_6}{4} ez_2^4 + \bar{o}(z^4)
		\]
		with some constant $a$. Finally, consider the expansion
		\[
		S_{11}^f = p_0 + p_{11}z_1^2 + p_{12}z_1z_2 + p_{22}z_2^2 + \bar{o}(z^2).
		\]
		By inserting the expansions for $S^f$ and $S_{11}^f$ into the equation
		\[
		S^f\frac{\partial^2 S^f}{\partial z_1^2} - \left(\frac{\partial S^f}{\partial z_1}\right)^2 = -2S^fS^f_{11}  + f_6S^fS_{12}^f
		\]
		and equating the part of order $\le 4$ we get the identity
		\begin{multline*}
		(2 - 2p_0)z_1^2 + (2a - 2p_0a - 2p_{11})z_1^4 - 2p_{12}z_1^3z_2 - 2p_{22}z_1^2z_2^2 + \\+ \frac{f_5 p_0 - f_5}{3}z_1 z_2^3 + \frac{f_6p_0 - f_6}{2}z_2^4 = 0.
		\end{multline*}
		Considering coefficient of $z_1^2$ we immediately get $p_0 = 1$ and by substituting it back into equation we get $p_{11} = p_{12} = p_{22} = 0$.

	\end{proof}

    \printbibliography
\end{document}